\documentclass[final,leqno]{siamltex}

\usepackage{amssymb}
\usepackage{latexsym}
\usepackage{amsfonts}
\usepackage{amsmath}
\usepackage{amscd}     % commutative diagrams
\usepackage{mathrsfs}  % This package allows the use of script letters. Type \mathscr to invoke math script.
\usepackage{amsbsy}
\usepackage{graphicx}
\usepackage{color}
\usepackage{color}

\newtheorem{example}{Example}

\newtheorem{remark}{Remark}

\newcommand{\K}{\mathbb{K}}
\newcommand{\N}{\mathbb{N}}

\newcommand{\R}{\mathbb{R}}

\newcommand{\X}{\mathbb{X}}
\newcommand{\Y}{\mathbb{Y}}
\newcommand{\Z}{\mathbb{Z}}

\newcommand{\cF}{\mathcal{F}}

\newcommand{\calH}{\mathcal{H}}

\newcommand{\cW}{\mathcal{W}}

\newcommand{\eps}{\varepsilon}

\newcommand{\be}{\begin{equation}}
\newcommand{\ee}{\end{equation}}

\newcommand{\st}{\,|\,}

\newcommand{\bS}{{\boldsymbol{S}}}

\newcommand{\blambda}{{\boldsymbol{\lambda}}}
\newcommand{\bmu}{{\boldsymbol{\mu}}}

\newcommand{\loc}{\mathrm{loc}}

\newcommand{\argmin}{\operatorname{argmin}}

\title{Numerics and Fractals\thanks{This research was partially supported under Australian Research Council's Discovery 
             Projects funding scheme (project number DP130101738) and the Technische 
             Universit\"at M\"unchen – Institute for Advanced Study, funded by the German Excellence Initiative.} }

\author{Michael F. Barnsley\thanks{Mathematical Sciences Institute, The Australian National University, Canberra, ACT, Australia}
\and Markus Hegland\thanks{Mathematical Sciences Institute, The Australian National University, Canberra, ACT, Australia and Institute for Advanced Study, Technische
Universit\"at M\"unchen, Germany}
\and Peter Massopust\thanks{Centre of Mathematics, M6, Technische Universit\"at M\"unchen, Germany}}

\begin{document}

\maketitle

\begin{abstract}
  Local iterated function systems are an important generalisation of the standard (global) iterated function systems (IFSs). For a particular class of mappings, their fixed points are the graphs of local fractal functions and these functions themselves are known to be the 
  fixed points of an associated Read-Bajactarevi\'c operator. This paper establishes existence and properties of local fractal
  functions and discusses how they are computed. In particular, it is shown that piecewise polynomials are a
  special case of local fractal functions. Finally, we develop a method to compute the components of a local IFS
  from data or (partial differential) equations.
\end{abstract}



\pagestyle{myheadings}
\thispagestyle{plain}
\markboth{BARNSLEY, HEGLAND AND MASSOPUST}{NUMERICS AND FRACTALS}

\section{Introduction}

Contractive operators on function spaces are important for the
development of both the theory and algorithms for the solution of
integral and differential equations. They are used in the theory of
elliptic partial differential equations, Fredholm integral equations of the second kind,
Volterra integral equations, and ordinary differential equations. This
is just a small selection of instances were they appear in mathematics.
Contractive operators are fundamental for the development of iterative
solvers in general and wavelet-based solvers for elliptic
problems \cite{CohDD01} in particular.

One class of contractive operators is defined on the graphs of
functions using a special kind of iterated function system (IFS). The fixed point 
of such an IFS is the graph of a fractal function. There is a vast 
literature on IFSs, see for example the recent review by the first 
author~\cite{B0}.
Computationally, IFSs are used in Computer Graphics in refinement
methods which effectively compute points on curves and surfaces~\cite{CavDM91}.
They are also used to compute function values of piecewise polynomial
functions and wavelets. In fact, it can be shown that these applications
use a variant of IFSs where the iterated functions are defined locally~\cite{barnhurd}.
These local IFSs and, in particular, their computational application
are the topic of the following discussion. In this first manuscript we will
mostly consider functions of one real variable in the examples. Functions
of multiple variables are planned to be covered in a future paper.

The remaining part of this introduction will provide some further background
and motivation for our approach to utilise IFSs or local IFSs in computations.
In the second section we introduce and review local IFSs. 
The third section applies local IFSs to graphs of functions
to define local fractal functions. It will be seen that these functions are
the fixed points of a Read--Bajactarevi\'c  (RB) operator. (See also \cite{mas1} for the use of such
operators in the theory of (global) fractal functions.) Section 4 provides a 
reformulation of the RB operator in terms of matrices acting on vectors of function
values over grids. Several examples of local fractal functions are then displayed.
In Section 5 we discuss the  important case of polynomials and their RB operators.
In a penultimate section we discuss the determination of (approximate) iterated
function systems both from data and from functional equations such as partial
differential equations. We conclude this discussion with some general remarks 
and in particular with a connection between fractals and the active research area
of tensor approximation.

\subsection{Fractals and numerics}

One can show that graphs of piecewise polynomial functions can be written as the 
fixed points of local IFSs. Thus the popular finite element method approximates 
solutions of PDEs with particular fractal functions. However, numerical methods
do not usually use IFSs directly. Exceptions are the subdivision schemes
used in computer graphics where (local) IFSs are employed mostly for the 
representation of smooth curves and surfaces. 

We suggest the construction and use of IFSs for the solution of PDEs. This is done 
by choosing an initial IFS and then changing it iteratively until it approximates 
a desired function given by either data or functional (e.g. partial differential) 
equations. We use ideas based on the collage theorem to fit a given function class 
and refine the domains of the IFS if necessary.

In the following we will discuss the numerical application of local IFSs
which is not based on a basis of a linear space but on the IFS itself.
We anticipate that this approach has the following advantages over 
approaches that are based on a linear basis:
\begin{itemize}
\item The same approach can be used to approximate and solve PDEs on very general
  grids defined by IFSs including fractal sets.
\item Visualisation and numerical solutions are computed simultaneously
  and can be done on the same or on neighboring processors of a multiprocessor
  system such that communication overhead may be reduced.
\item
  Dimensionality is handled much more flexibly in fractals -- for
  example, one may use 1D solvers for higher-dimensional problems.
\item
  We can at the same time adapt the basis functions (or frames) as well
  as solving the problem. Searches over large collections of
  dictionaries of finite dimensional approximation spaces can be done
  locally during the solution.
\item
  The computational complexity is bounded by the resolution one
  requires.
\item
  Adaptivity is naturally included as in wavelet-based methods and is
  a consequence of the iteration  -- one application of the IFS reduces
  the finest scale.
\item
  Convergence of the method can be controlled with few parameters and is
  driven by the convergence of the IFS.
\item
  The theory is based on the theory for fractals and IFSs which is well
  established. In addition, there has been a lot of work on wavelets and
  subdivision schemes which provides further firm foundations.
\end{itemize}

\subsection{The Collage Theorem}

While it is usually assumed that the iterated function system (IFS) is given, a
very important class of methods used in image processing determines the
IFS from its fixed point. An important result used
here is the Collage Theorem \cite{B}. For the purposes of self-containment, we state
this theorem below.
\begin{theorem}
Let $(\X,d_\X)$ be a complete metric space. Denote by $(\calH (\X), d_\calH)$ the associated complete metric space based on the hyperspace of nonempty compact subsets of $\X$ endowed with the Hausdorff metric $d_\calH$. Let $M\in\calH(\X)$ and $\varepsilon > 0$ be given. Suppose that $\cF := \{\X; f_1, \ldots, f_N\}$ is a contractive IFS such that
\[
d_\calH \left(M, \;\bigcup_{i=1}^N f_i (M) \right) < \varepsilon.
\]
Then
\[
d_\calH (M, A) < \frac{\varepsilon}{1-s},
\]
where $A$ is the attractor of the IFS and  and $s := \max\{\mathrm{Lip}\,f_i\st i = 1, \ldots, N\}$.
\end{theorem}

It has been demonstrated that approaches that are based on the Collage Theorem lead to very efficient image compression methods. The interested reader is referred to \cite{barnhurd,fisher} for methodologies and to \cite{vrscay} for a summery of fractal-type approaches in an analytical setting. Note that the application of an IFS starts with points on a large scale and then
moves to finer scales. This is very similar to some multigrid methods
and wavelet methods.

\section{Local Iterated Function Systems}
The concept of \textit{local} iterated function system is a generalization of an IFS as defined in \cite{B} and was first introduced in \cite{barnhurd}.

In the following, $(\X, d_\X)$ denotes a complete metric space with metric $d_\X$ and $\N := \{1, 2, 3,\ldots\}$ the set of positive integers. 

\begin{definition}\label{localIFS}
Let $N\in \N$ and let $\N_N := \{1, \ldots, N\}$. Suppose $\{\X_i \st i \in \N_N\}$ is a family of nonempty subsets of $\X$. Further assume that for each $\X_i$ there exists a continuous mapping $f_i:\X_i\to\X$, $i\in \N_N$. Then $\cF_{\loc} := \{\X; (\X_i, f_i)\st i \in \N_N\}$ is called a \textbf{local iterated function system} (local IFS).
\end{definition}

Note that if each $\X_i = \X$, then Definition \ref{localIFS} coincides with the usual definition of a standard (global) IFS on a complete metric space. However, the possibility of choosing the domain for each continuous mapping $f_i$ different from the entire space $\X$ adds additional flexibility as will be recognized in the sequel.

A mapping $f: U\subset \X \to \X$ is called \textbf{contractive on $U$} or \textbf{a contraction on $U$} if there exists a constant $s\in [0,1)$ so that
\[
d_\X (f(x_1), f(x_2)) \leq s \,d_\X(x_1, x_2), \quad\forall\,x_1,x_2\in \X.
\]

\begin{definition}
A local IFS $\cF_{\loc}$ is called \textbf{contractive} if there exists a metric $d'$ equivalent to $d_\X$ with respect to which 
all functions $f\in \cF_{\loc}$ are contractive (on their respective domains).
\end{definition}

Let $2^\X := \{S \st S \subseteq \X\}$ be the power set of $\X$. With a local IFS we associate a set-valued operator $\cF_\loc : 2^\X \to 2^\X$ by setting
\be\label{hutchop}
\cF_\loc(S) := \bigcup_{i=1}^N f_i (S\cap \X_i).
\ee

Here $f_i(S\cap \X_i) = \{f_i(x)\mid x\in S\cap \X_i\}$.
By a slight abuse of notation, we use the same symbol for a local IFS and its associated operator.

\begin{definition}
A subset $A\in 2^\X$ is called a \textbf{local attractor} for the local IFS $\{\X; (\X_i, f_i)\st i \in \N_N\}$ if
\be\label{attr}
A = \cF_\loc (A) = \bigcup_{i=1}^N f_i (A\cap \X_i).
\ee
\end{definition}
In \eqref{attr} we allow for $A\cap \X_i$ to be the empty set. Thus, every local IFS has at least one local attractor, namely $A = \emptyset$. However, it may also have many distinct ones. In the latter case, if $A_1$ and $A_2$ are distinct local attractors, then $A_1\cup A_2$ is also a local attractor. Hence, there exists a largest local attractor for $\cF_\loc$, namely the union of all distinct local attractors. We refer to this largest local attractor as {\em the} local attractor of a local IFS $\cF_\loc$.

We remark that there exists an alternative definition for \eqref{hutchop}. For given functions $f_i$ which are only defined on
$\X_i$ one could introduce set functions (which will also be called $f_i$) which are defined on $2^\X$ by
\[
f_i (S) := \begin{cases} f_i (S\cap \X_i), & S\cap \X_i\neq \emptyset;\\ \emptyset, & S\cap \X_i = \emptyset,\end{cases}  \qquad i\in \N_N, \; S\in 2^\X.
\]
On the left-hand side $f_i(S\cap \X_i)$ is the set of values of the original $f_i$ as in the previous definition. 
This extension of a given function $f_i$ to sets $S$ which include elements which are not in the domain of $f_i$ basically
just ignores those elements. In the following we will assume this definition of the set function $f_i$ to be used. 

In the case where $\X$ is compact and the $\X_i$, $i\in \N_N$ closed, i.e., compact in $\X$, and where the local IFS $\{\X; (\X_i, f_i)\st i \in \N_N\}$ is contractive, the local attractor may be computed as follows. Let $K_0:= \X$ and set
\[
K_n := \cF_\loc (K_{n-1}) = \bigcup_{i\in \N_N} f_i (K_{n-1}\cap \X_i), \quad n\in \N.
\]
Then $\{K_n\st n\in\N_0\}$ is a decreasing nested sequence of compact sets. \textit{If} each $K_n$ is nonempty, then by the Cantor Intersection Theorem,
\[
K:= \bigcap_{n\in \N_0} K_n \neq \emptyset.
\]
Using \cite[Proposition 3 (vii)]{lesniak}, we see that
\[
K = \lim_{n\to\infty} K_n,
\]
where the limit is taken with respect to the Hausdorff metric on $\calH(\X)$. This implies that
\[
K = \lim_{n\to\infty} K_n = \lim_{n\to\infty} \bigcup_{i\in \N_N} f_i (K_{n-1}\cap \X_i) = \bigcup_{i\in \N_N} f_i (K\cap \X_i) = \cF_\loc (K). 
\]
Thus, $K = A_\loc$. A (mild) condition guaranteeing that each $K_n$ is nonempty is that $f_i(\X_i) \subset\X_i$, $i\in \N_N$. (See also \cite{barnhurd}.)

In the above setting where the $f_i$ have been extended to $2^\X$, one can derive a relation between the local attractor $A_\loc$ of a contractive local IFS $\{\X; (\X_i, f_i)\st i \in \N_N\}$ and the (global) attractor $A$ of the associated (global) IFS $\{\X; f_i\st i \in \N_N\}$ where the extensions of $f_i$ to all sets are defined as above. To this end, let the sequence $\{K_n\st n\in \N_0\}$ be defined as above. The unique attractor $A$ of the IFS $\cF:= \{\X; f_i\st i \in \N_N\}$ is obtained as the fixed point of the set-valued map $\cF: \calH(\X)\to \calH(\X)$, 
\be\label{setvalued}
\cF (B) = \bigcup_{i\in \N_N} f_i (B),
\ee
where $B\in \calH(\X)$. If the IFS $\cF$ is contractive, then the set-valued mapping \eqref{setvalued} is contractive on $\calH(\X)$ (with respect to the Hausdorff metric) and its fixed point can be obtained as the limit of the sequence of sets $\{A_n\st n\in \N_0\}$ with $A_0 := \X$ and 
\[
A_n := \cF(A_{n-1}), \quad n\in \N.
\]
Note that $K_0 = A_0 = \X$ and, assuming that $K_{n-1}\subseteq A_{n-1}$, $n\in\N$, it follows by induction that
\begin{align*}
K_n & = \bigcup_{i\in \N_N} f_i (K_{n-1}\cap \X_i) \subseteq \bigcup_{i\in \N_N} f_i (K_{n-1}) \subseteq \bigcup_{i\in \N_N} f_i (A_{n-1}) = A_n.
\end{align*}
Hence, upon taking the limit with respect to the Hausdorff metric as $n\to\infty$, we obtain $A_\loc \subseteq A$. This proves the next result.

\begin{proposition}
Let $\X$ be a compact metric space and let $\X_i$, $i\in \N_N$, be closed, i.e., compact in $\X$. Suppose that the local IFS $\cF_\loc := \{\X; (\X_i, f_i)\st i \in \N_N\}$ and the IFS $\cF:=\{\X; f_i\st i \in \N_N\}$ are both contractive. Then the local attractor $A_\loc$ of $\cF_\loc$ is a subset of the attractor $A$ of $\cF$. 
\end{proposition}

Contractive local IFSs are point-fibered if $\X$ is compact and the $\X_i$, $i\in \N_N$, are closed. To show this, define the code space of a local IFS by $\Omega:= \prod_{n\in\N}\N_N$ and endow it with the product topology $\mathfrak{T}$. It is known that $\Omega$ is metrizable and that $\mathfrak{T}$ is induced by the Fr\'echet metric $d_F: \Omega\times\Omega\to \R$,
\[
d_F(\sigma,\tau) := \sum_{n\in \N} \frac{|\sigma_n - \tau_n|}{(N+1)^n},
\]
where $\sigma = (\sigma_1\ldots\sigma_n\ldots)$ and $\tau = (\tau_1\ldots\tau_n\ldots)$. (As a reference, see for instance \cite{Engelking}, Theorem 4.2.2.) The elements of $\Omega$ are called codes.

Define a set-valued mapping $\gamma :\Omega \to \K(\X)$, where $\K(\X)$ denotes the hyperspace of all compact subsets of $\X$, by
\[
\gamma (\sigma) := \bigcap_{n=1}^\infty f_{\sigma_1}\circ \cdots \circ f_{\sigma_n} (\X),
\]
where $\sigma = (\sigma_1\ldots\sigma_n\ldots)$. Then $\gamma (\sigma)$ is point-fibred, i.e., a singleton. Moreover, in this case, the local attractor $A$ equals $\gamma(\Omega)$. (For details regarding point-fibred IFSs, we refer the interested reader to \cite{K}, Chapters 3--5.) 

\begin{example}
Let $\X := [0,1]\times [0,1]$ and suppose that $0 < x_2 < x_1 < 1$ and $0 < y_2 < y_1 < 1$. Define
\[
\X_1 := [0,x_1]\times [0,y_1]\qquad\text{and}\qquad \X_2 := [x_2,1]\times [y_2,1].
\]
Furthermore, let $f_i:\X_i \to \X$, $i=1,2$, be given by
\[
f_1(x,y) := (s_1 x, s_1 y)\quad\text{and}\quad f_2(x,y) := (s_2 x + (1-s_2) x_2, s_2 y + (1-s_2)y_2),
\]
respectively, where $s_1,s_2\in [0,1)$.

The (global) IFS $\{\X; f_1, f_2\}$ has as its unique attractor the line segment $A = \{(x, \frac{y_2}{x_2}\, x)\st 0\leq x \leq x_2\}$. The local attractor $A_\loc$ of the local IFS $\{\X; (\X_1, f_1), (\X_2, f_2)\}$ is the union of the fixed point $(0,0)$ of $f_1$ and the fixed point $(x_2,y_2)$ of $f_2$.
\end{example}
\section{Local Fractal Functions}\label{locfracfun}
In this section, we exhibit a class of special attractors of local IFSs, namely local attractors that are the graphs of bounded functions. These functions will be called \textbf{local fractal functions}. We prove that the set of discontinuities of these bounded functions is countably infinite and we derive conditions under which local fractal functions are elements of the Lebesgue spaces $L^p$. 

To this end, we assume that $1 < N\in \N$ and set $\N_N := \{1, \ldots, N\}$. Let $\X$ be a nonempty connected set and $\{\X_i \st i \in\N_N\}$ a family of nonempty connected subsets of $\X$. Suppose $\{u_i : \X_i\to \X \st i \in \N_N\}$ is a family of bijective mappings with the property that
\begin{enumerate}
\item[(P)] $\{u_i(\X_i)\st i \in\N_N\}$ forms a (set-theoretic) partition $\X$, i.e., $\X = \bigcup_{i=1}^N u_i(\X_i)$ and $u_i(\X_i)\cap u_j(\X_j) = \emptyset$, for all $i\neq j\in \N_N$.
\end{enumerate}
\noindent
Now suppose that $(\Y,d_\Y)$ is a complete metric space with metric $d_\Y$. A mapping $f:\X\to \Y$ is called \textbf{bounded} (with respect to the metric $d_\Y$) if there exists an $M> 0$ so that for all $x_1, x_2\in \X$, $d_\Y(f(x_1),f(x_2)) < M$.

Denote by $B(\X, \Y)$ the set
\[
B(\X, \Y) := \{f : \X\to \Y \st \text{$f$ is bounded}\}.
\]
Endowed with the metric 
\[
d(f,g): = \displaystyle{\sup_{x\in \X}} \,d_\Y(f(x), g(x)),
\] 
$(B(\X, \Y), d)$ becomes a complete metric space. Similarly, we define $B(\X_i, \Y)$, $i \in\N_N$.

\begin{remark}
Note that under the usual addition and scalar multiplication of functions, the spaces $B(\X_i,\Y)$ and $B(\X,\Y)$ become metric linear spaces. A \textit{metric linear space} is a vector space endowed with a metric under which the operations of vector addition and scalar multiplication are continuous.
\end{remark}

For $i \in \N_N$, let $v_i: \X_i\times \Y \to \Y$ be a mapping that is uniformly contractive in the second variable, i.e., there exists an $\ell\in [0,1)$ so that for all $y_1, y_2\in \Y$
\be\label{scon}
d_\Y (v_i(x, y_1), v_i(x, y_2)) \leq \ell\, d_\Y (y_1, y_2), \quad\forall x\in \X.
\ee
Define a Read-Bajactarevi\'c (RB) operator $\Phi: B(\X,\Y)\to \Y^{\X}$ by
\be\label{RB}
\Phi f (x) := \sum_{i=1}^N v_i (u_i^{-1} (x), f_i\circ u_i^{-1} (x))\,\chi_{u_i(\X_i)}(x), 
\ee
where $f_i := f\vert_{\X_i}$ and 
$$
\chi_M (x) := \begin{cases} 1, & x\in M\\ 0, & x\notin M\end{cases}.
$$
Note that $\Phi$ is well-defined and since $f$ is bounded and each $v_i$ contractive in the second variable, $\Phi f\in B(\X,\Y)$.

Moreover, by \eqref{scon}, we obtain for all $f,g\in B(\X, \Y)$ the following inequality:
\begin{align}\label{estim}
d(\Phi f, \Phi g) & = \sup_{x\in \X} d_\Y (\Phi f (x), \Phi g (x))\nonumber\\
& = \sup_{x\in \X} d_\Y (v(u_i^{-1} (x), f_i(u_i^{-1} (x))), v(u_i^{-1} (x), g_i(u_i^{-1} (x))))\nonumber\\
& \leq \ell\sup_{x\in \X} d_\Y (f_i\circ u_i^{-1} (x), g_i \circ u_i^{-1} (x)) \leq \ell\, d_\Y(f,g).
\end{align}
To simplify notation, we set $v(x,y):= \sum_{i=1}^N v_i (x, y)\,\chi_{\X_i}(x)$ in the above equation. In other words, $\Phi$ is a contraction on the complete metric space $B(\X,\Y)$ and, by the Banach Fixed Point Theorem, has therefore a unique fixed point $f^*$ in $B(\X,\Y)$. This unique fixed point will be called a {\em local fractal function}  $f^* = f^*_\Phi$ (generated by $\Phi$).

Next, we would like to consider a special choice for mappings $v_i$. To this end, we require the concept of an $F$-space. We recall that a metric $d:\Y\times\Y\to \R$ is called \textbf{complete} if every Cauchy sequence in $\Y$ converges with respect to $d$ to a point of $\Y$, and \textbf{translation-invariant} if $d(x+a,y+a) = d(x,y)$, for all $x,y,a\in \Y$.

\begin{definition}
A topological vector space $\Y$ is called an \textbf{$\boldsymbol{F}$-space} if its topology is induced by a complete translation-invariant metric $d$.
\end{definition}

Now suppose that $\Y$ is an $F$-space. Denote its metric by $d_\Y$. We define mappings $v_i:\X_i\times\Y\to \Y$ by
\be\label{specialv}
v_i (x,y) := \lambda_i (x) + S_i (x) \,y,\quad i \in \N_N,
\ee
where $\lambda_i \in B(\X_i,\Y)$ and $S_i : \X_i\to \R$ is a function.

If in addition we require that the metric $d_\Y$ is homogeneous, that is,
\[
d_\Y(\alpha y_1, \alpha y_2) = |\alpha| d_\Y(y_1,y_2), \quad \forall \alpha\in\R\;\forall y_1.y_2\in \Y,
\]
then $v_i$ given by \eqref{specialv} satisfies condition \eqref{scon} provided that the functions $S_i$ are bounded on $\X_i$ with bounds in $[0,1)$ for then
\begin{align*}
d_\Y (\lambda_i (x) + S_i (x) \,y_1,\lambda_i (x) + S_i (x) \,y_2) &= d_\Y(S_i (x) \,y_1,S_i (x) \,y_2) \\
& = |S_i(x)| d_\Y (y_1, y_2)\\
& \leq \|S_i\|_{\infty,\X_i}\, d_\Y (y_1, y_2)\\
& \leq s\,d_\Y (y_1, y_2).
\end{align*}
Here, $\|\bullet\|_{\infty, \X_i}$ denotes the supremum norm with respect to $\X_i$ and $s := \max\{\|S_i\|_{\infty,\X_i}\st$ $i\in \N_N\}$.

Thus, for a fixed set of functions $\{\lambda_1, \ldots, \lambda_N\}$ and $\{S_1, \ldots, S_N\}$, the associated RB operator \eqref{RB} has now the form
\[
\Phi f = \sum_{i=1}^N \lambda_i\circ u_i^{-1} \,\chi_{u_i(\X_i)} + \sum_{i=1}^N (S_i\circ u_i^{-1})\cdot (f_i\circ u_i^{-1})\,\chi_{u_i(\X_i)},
\]
or, equivalently,
\[
\Phi f_i\circ u_i = \lambda_i + S_i\cdot f_i, \quad \text{on $\X_i$, $\forall\;i\in\N_N$,}
\]
with $f_i = f\vert_{\X_i}$.

\begin{theorem}
Let $\Y$ be an $F$-space with homogeneous metric $d_\Y$. Let $\X$ be a nonempty connected set and $\{\X_i \st i \in\N_N\}$ a family of nonempty connected subsets of $\X$. Suppose $\{u_i : \X_i\to \X \st i \in \N_N\}$ is a family of bijective mappings satisfying property $\mathrm{(P)}$.

Let $\blambda := (\lambda_1, \ldots, \lambda_N)\in \underset{i=1}{\overset{N}{\times}} B(\X_i,\Y)$, and $\bS := (S_1, \ldots, S_N)\in \underset{i=1}{\overset{N}{\times}} B (\X_i,\R)$. Define a mapping $\Phi: \left(\underset{i=1}{\overset{N}{\times}} B(\X_i,\Y)\right)\times \left(\underset{i=1}{\overset{N}{\times}} B (\X_i,\R)\right) \times B(\X,\Y)\to B(\X,\Y)$ by
\be\label{eq3.4}
\Phi(\blambda)(\bS) f = \sum_{i=1}^N \lambda_i\circ u_i^{-1} \,\chi_{u_i(\X_i)} + \sum_{i=1}^N (S_i\circ u_i^{-1})\cdot (f_i\circ u_i^{-1})\,\chi_{u_i(\X_i)}.
\ee
If $\max\{\|S_i\|_{\infty,\X_i}\st i\in \N_N\} < 1$ then the operator $\Phi(\blambda)(\bS)$ is contractive on the complete metric space $B(\X, \Y)$ and its unique fixed point $f^*$ satisfies the self-referential equation
\be\label{3.4}
f^* = \sum_{i=1}^N \lambda_i\circ u_i^{-1} \,\chi_{u_i(\X_i)} + \sum_{i=1}^N (S_i\circ u_i^{-1})\cdot (f^*_i\circ u_i^{-1})\,\chi_{u_i(\X_i)},
\ee
or, equivalently
\be
f^*\circ u_i = \lambda_i + S_i\cdot f^*_i, \quad \text{on $\X_i$, $\forall\;i\in\N_N$,}
\ee
where $f^*_i = f^*\vert_{\X_i}$.

This fixed point $f^*$ is called a \textbf{local fractal function}.
\end{theorem}
\begin{proof}
The statements follow directly from the considerations preceding the theorem.
\end{proof}

\begin{remark}
Note that the local fractal function $f^*$ generated by the operator defined by \eqref{eq3.4} does not only depend on the family of subsets $\{\X_i \st i \in \N_N\}$ but also on the two $N$-tuples of bounded functions $\blambda\in \underset{i=1}{\overset{N}{\times}} B(\X_i,\Y)$, and $\bS\in \underset{i=1}{\overset{N}{\times}} B (\X_i,\R)$. The fixed point $f^*$ should therefore be written more precisely as $f^* (\blambda)(\bS)$. However, for the sake of notational simplicity, we usually suppress this dependence for both $f^*$ and $\Phi$.
\end{remark}

The following result found in \cite{GHM} and in more general form in \cite{M97} is the extension to the setting of local fractal functions.
\begin{theorem}\label{thm3.3}
The mapping $\blambda \mapsto f^*(\blambda)$ defines a linear isomorphism from $\underset{i=1}{\overset{N}{\times}} B(\X_i,\Y)$ to $B(\X,\Y)$.
\end{theorem}
\begin{proof}
Let $\alpha, \beta \in\R$ and let $\blambda, \bmu\in \underset{i=1}{\overset{N}{\times}} B(\X_i,\Y)$. Injectivity follows immediately from the fixed point equation \eqref{3.4} and the uniqueness of the fixed point: $\blambda = \bmu$ $\Longleftrightarrow$ $f^*(\blambda) = f^*(\bmu)$, .

Linearity follows from \eqref{3.4}, the uniqueness of the fixed point and injectivity: 
\begin{align*}
f^*(\alpha\blambda + \beta \bmu) & = \sum_{i=1}^N (\alpha\lambda_i + \beta \mu_i) \circ u_i^{-1} \,\chi_{u_i(\X_i)}\\
& \qquad  + \sum_{i=1}^N (S_i\circ u_i^{-1})\cdot (f_i^*(\alpha\blambda + \beta \bmu)\circ u_i^{-1})\,\chi_{u_i(\X_i)}
\end{align*}
and
\begin{align*}
\alpha f^*(\blambda) + \beta f^*(\bmu) & = \sum_{i=1}^N (\alpha\lambda_i + \beta \mu_i) \circ u_i^{-1} \,\chi_{u_i(\X_i)}\\
& \qquad  + \sum_{i=1}^N (S_i\circ u_i^{-1})\cdot (\alpha f_i^*(\blambda) + \beta f_i^*(\bmu))\circ u_i^{-1})\,\chi_{u_i(\X_i)}.
\end{align*}
Hence, $f^*(\alpha\blambda + \beta \bmu) = \alpha f^*(\blambda) + \beta f^*(\bmu)$.

For surjectivity, we define $\lambda_i := f^*\circ u_i - S_i \cdot f^*$, $i\in \N_N$. Since $f^*\in B(\X,\Y)$, we have $\blambda\in \underset{i=1}{\overset{N}{\times}} B(\X_i,\Y)$. Thus, $f^*(\blambda) = f^*$.
\end{proof}

We may construct local fractal functions on spaces other than $B(\X,\Y)$. To this end, we assume again that the functions $v_i$ are given by \eqref{specialv} and that $\X := [0,1]$ and $\Y := \R$. We consider the metric on $\R$ and $[0,1]$ as being induced by the $L^1$-norm. Note that endowed with this norm $B([0,1],\R)$ becomes a Banach space.

%Recall that the Lebesgue spaces $L^p [0,1]$, $1\leq p\leq \infty$, are obtained as the completion of the space $C[0,1]$ of real-valued continuous functions on $[0,1]$ with respect to the $L^p$-norm
%\[
%\|f\|_{L^p} := \left(\int_{[0,1]} |f(x)|^p \,dx\right)^{1/p}.
%\]

We have the following result for RB-operators defined on the Lebesgue spaces $L^p[0,1]$, $1\leq p \leq \infty$.

\begin{theorem}
Let $1<N\in \N$ and suppose that $\{\X_i \st i \in \N_N\}$ is a family of half-open intervals of $[0,1]$. Further suppose that $P := \{x_0 := 0 < x_1 < \cdots < x_N := 1\}$ is a partition of $[0,1]$ and that $\{u_i \st i \in\N_N\}$ is a family of affine mappings from $\X_i$ onto $[x_{i-1}, x_i)$, $i = 1, \ldots, N-1$, and from $\X_N^+ := \X_N\cup u_N^{-1}(1-)$ onto $[x_{N-1},x_N]$, where $u_N$ maps $\X_N$ onto $[x_{N-1}, x_N)$. 

The operator $\Phi: L^p [0,1]\to \R^{[a,b]}$ defined by
\be\label{Phi}
\Phi g := \sum_{i=1}^N (\lambda_i \circ u_i^{-1})\,\chi_{u_i(\X_i)} + \sum_{i=1}^N (S_i\circ u_i^{-1})\cdot (g_i\circ u_i^{-1})\,\chi_{u_i(\X_i)},
\ee
where $g_i = g\vert_{\X_i}$, $\lambda_i\in L^p (\X_i, [0,1])$ and $S_i\in L^\infty (\X_i, \R)$, $i \in\N_N$, maps $L^p [0,1]$ into itself. Moreover, if 
\be\label{condition}
\begin{cases}
\left(\displaystyle{\sum_{i=1}^N}\, a_i \,\|S_i\|_{\infty, \X_i}^p\right)^{1/p} < 1, & p\in[1,\infty);\\ \\
\max\left\{\|S_i\|_{\infty,\X_i}\st i\in\N_N\right\} < 1, & p = \infty,
\end{cases}
\ee
where $a_i$ denotes the Lipschitz constant of $u_i$, then $\Phi$ is contractive on $L^p [0,1]$ and its unique fixed point $f^*$ is an element of $L^p [0,1]$.
\end{theorem}

\begin{proof}
Note that under the hypotheses on the functions $\lambda_i$ and $S_i$ as well as the mappings $u_i$, $\Phi f$ is well-defined and an element of $L^p[0,1]$. It remains to be shown that under condition \eqref{condition}, $\Phi$ is contractive on $L^p[0,1]
$. 

To this end, let $g,h \in L^p [0,1]$ and let $p\in [0,\infty)$. Then
\begin{align*}
\|\Phi g - \Phi h\|^p_{p} & = \int\limits_{[0,1]} |\Phi g (x) - \Phi h (x)|^p dx\\
& = \int\limits_{[0,1]} \left|\sum_{i=1}^{N} (S_i\circ u_i^{-1})(x) [(g_i\circ u_i^{-1})(x) - (h_i\circ u_i^{-1})(x)]\,\chi_{u_i(\X_i)}(x)\right|^p\, dx\\
& = \sum_{i=1}^{N}\,\int\limits_{[x_{i-1},x_i]}\left| (S_i\circ u_i^{-1})(x) [(g_i\circ u_i^{-1})(x) - (h_i\circ u_i^{-1})(x)]\right|^p\,dx\\
& = \sum_{i=1}^{N}\,a_i\,\int\limits_{\X_i} \left| S_i (x) [g_i(x)- h_i(x)]\right|^p\,dx\\
&  \leq \sum_{i=1}^{N}\,a_i\,\|S_i\|^p_{\infty, \X_i}\,\int\limits_{\X_i} \left| g_i(x) - h_i(x)\right|^p\,dx = \sum_{i=1}^{N}\,a_i\,\|S_i\|^p_{\infty, \X_i}\,\|f_i - g_i\|^p_{p,\X_i}\\
& = \sum_{i=1}^{N}\,a_i\,\|S_i\|^p_{\infty, \X_i}\,\|g_i - h_i\|^p_{p} \leq \left(\sum_{i=1}^{N}\,a_i\,\|S_i\|^p_{\infty, \X_i}\right) \|g - h\|^p_{p}.
\end{align*}
Now let $p= \infty$. Then
\begin{align*}
\|\Phi g - \Phi h\|_{\infty} & = \left\|\sum_{i=1}^{N} (S_i\circ u_i^{-1})(x) [(g_i\circ u_i^{-1})(x) - (h_i\circ u_i^{-1})(x)]\,\chi_{u_i(\X_i)}(x)\right\|_\infty\\
& \leq \max_{i\in\N_N}\,\left\| (S_i\circ u_i^{-1})(x) [(g_i\circ u_i^{-1})(x) - (h_i\circ u_i^{-1})(x)]\right\|_{\infty,\X_i}\\
& \leq \max_{i\in\N_N}\|S_i\|_{\infty,\X_i} \left\|g_i - h_i]\right\|_{\infty,\X_i} = \max_{i\in\N_N}\|S_i\|_{\infty,\X_i} \left\|g_i - h_i]\right\|_{\infty}\\
&  \leq \left(\max_{{i\in\N_N}}\,\|S_i\|_{\infty,\X_i}\right) \left\|g - h]\right\|_{\infty}
\end{align*}
These calculations prove the claims.
\end{proof}

\begin{remark}
The proof of the theorem shows that the conclusions also hold under the assumption that the family of mappings $\{u_i: \X_i\to \X\st i\in \N_N\}$ is generated by the following functions.
\begin{enumerate}
\item[$\mathrm{(i)}$] Each $u_i$ is a bounded diffeomorphism of class $C^k$, $k\in \N\cup\{\infty\}$, from $\X_i$ to $[x_{i-1}, x_i)$ (obvious modification for $i = N$). In this case, the $a_i$'s are given by $a_i = \sup\{\left\vert \frac{du_i}{dx} (x)\right\vert\st x$ $\in \X_i\}$, $i\in\N_N$.
\item[$\mathrm{(ii)}$] Each $u_i$ is a bounded invertible function in $C^\omega$, the class of real-analytic functions from $\X_i$ to $[x_{i-1}, x_i)$ and its inverse is also in $C^\omega$. (Obvious modification for $i = N$.) The $a_i$'s are given as above in item $\mathrm{(i)}$.
\end{enumerate}
\end{remark}

Next we investigate the set of discontinuities of the fixed point $f^*$ of the RB-operator \eqref{Phi}.

\begin{theorem}\label{discont}
Let $\Phi$ be given as in \eqref{Phi}. Assume that for all $i\in \N_N$ the $u_i$ are contractive and the $\lambda_i$ are continuous on $\overline{\X_i}$. Further assume that condition \eqref{condition} is satisfied for $p=\infty$ and that the fixed point $f^*$ is bounded everywhere. Then the set of discontinuities of $f^*$ is at most countably infinite.
\end{theorem}

\begin{proof}
Let $f$ be a real-valued function and $U$ a nonempty open interval contained in its domain. The oscillation of $f$ on $U$ is defined as 
\[
\omega (f; U) := \sup_{x\in U} f(x) - \inf_{x\in U} f(x) = \sup_{x_1,x_2\in U} |f(x_1) - f(x_2)|,
\] 
and the oscillation of a function $f$ at a point $x_0$ inside an open interval contained in its domain is defined by
\[
\omega (f; x_0) := \lim_{\delta\to 0} \omega (f; (x_0 - \delta, x_0 + \delta)), \quad \delta > 0.
\]
The Banach Fixed Point Theorem implies that we may start with any bounded function, say $f_0 = \chi_{[0,1]}$, to construct a sequence of iterates $f_n := \Phi f_{n-1}$, $n\in \N$, which under the given hypotheses, converge in the $L^\infty$--norm to the fixed point $f^*$. 

Each iterate $f_n$ may have finite jump discontinuities at the interior knots $\{x_j \st j = 1, \ldots, N-1\}$ of the partition $P$ and also at the images $u_{i_1}\circ u_{i_2}\circ \cdots \circ u_{i_{-1}}(x_j)$ of the interior knots. The number of possible discontinuities at level $n$ is bounded above by $N^{n-1} (N-1)$ since the sets $\X_i$ may only contain a subset of the interior knots. Denote by $E_n$ the finite set of all finite jump discontinuities at level $n$ and let $E:= \bigcup_{n\in \N} E_n$. Note that $E$ is at most countably infinite.

Let $x\in [0,1]\setminus E$ and let $\eps > 0$. The fixed point equation for $f^*$,
\[
f^*(u_i(x))  = \lambda_i (x) + S_i(x) f^*_i(x), \quad x\in\X_i,
\]
implies that for all intervals $I \subset \X_i$,
\[
\omega (f^*; u_i(I)) \leq s\,\omega (f^*; I) + \Lambda\,|I|,
\]
where $s:= \max\{\|S_i\|_{\infty,\X_i}\st i\in \N_N\} < 1$ and $\Lambda = \max_{i\in \N_N} \sup_{x\in \X_i} |\lambda_i (x)|$. Hence, for any finite code $\sigma |K:= \sigma_1\sigma_2\cdots\sigma_K\in \Omega' := \bigcup_{m=0}^\infty \N_N^m$ of length $K\in \N$, we have that
\begin{align}\label{estosc}
\omega (f^*; u_{\sigma|K} (I)) & \leq s^K \omega (f^*; I)\nonumber\\
& \qquad + \Lambda\,|I| \left(a_{\sigma_2\cdots\sigma_K} + s a_{\sigma_3\cdots\sigma_K} + \cdots + s^{K-2}a_K + s^{K-1}\right)\nonumber\\
& \leq s^K \omega (f^*; I) + \Lambda\,|I| \left(a^{K-1} + s a^{K-2} + \cdots + s^{K-2}a + s^{K-1}\right)\nonumber\\
& \leq s^K \omega (f^*; I) + \Lambda\,|I|\,\frac{a^K}{|a-s|}.
\end{align}
for all intervals $I\subset\X_i$. Here, $a:= \max\{a_i\st i\in\N_N\} < 1$.

Note that $\{\X; (\X_i, u_i)\st i \in \N_N\}$ is a contractive local IFS with attractor $[0,1]$. As $\{\X; (\X_i, u_i)\st i \in \N_N\}$ is point-fibered, there exists a code $\sigma\in \Omega = \N_N^\infty$ such that 
$$
\gamma (\sigma) = \{x\} = \bigcap_{k\in \N} u_{\sigma |k}(\X).
$$
Given any $K\in\N$ there exists a nonempty compact interval $I_K$ such that
\[
x\in I_K \subset \bigcap_{k=1}^K u_{\sigma |K} (\X).
\]
The length $|I_K|$ of $I_K$ is bounded above by $a^K$. Set $J := u_{\sigma| K}^{-1} (I_K)$, where $u^{-1}_{\sigma |K} := u_{\sigma_K}^{-1} \circ \cdots \circ u_{\sigma_1}^{-1}$.

Using \eqref{estosc} we obtain
\[
\omega (f^*; I_K) = \omega (f^*; u_{\sigma|K} (J)) \leq s^K \omega (f^*; J) + \Lambda\,|J|\,\frac{a^K}{|a-s|}.
\]
Since $f^*$ is bounded on $[0,1]$, $|J|\leq 1$, and $a_K\to 0$ as $K\to\infty$, we can choose a $K$ large enough so that $s^K \omega (f^*; J) < \eps/2$ and $\Lambda\,|J|\,a^K/|a-s| < \eps/2$. Thus, $\omega (f^*; I_K) < \eps$, which proves the continuity of $f^*$ at all points in $[0,1]\setminus E$ and completes the proof.
\end{proof}

\begin{corollary}
Under the assumptions of Theorem \ref{discont}, the fixed point $f^*$ of $\Phi$ is Riemann-integrable over $[0,1]$.
\end{corollary}

\begin{proof}
This is a direct consequence of the above theorem and, for instance, Theorem 7.5 in \cite{oxtoby}.
\end{proof}

Next, we exhibit the relation between the graph $G$ of the fixed point $f^*$ of the operator $\Phi$ given by \eqref{RB} and the local attractor of an associated contractive local IFS. To this end, we need to require that $\X$ is a closed subset of a complete metric space. Consider the complete metric space $\X\times\Y$ and define mappings $w_i:\X_i\times\Y\to\X\times\Y$ by
\[
w_i (x, y) := (u_i (x), v_i (x,y)), \quad i\in \N_N.
\]
Assume that the mappings $v_i: \X_i\times \Y\to \Y$ in addition to being uniformly contractive in the second variable are also uniformly Lipschitz continuous in the first variable, i.e., that there exists a constant $L > 0$ so that for all $y\in \Y$,
\[
d_\Y(v_i(x_1, y),v_i(x_2, y)) \leq L \, d_\X (x_1,x_2), \quad\forall x_1, x_2\in \X_i,\quad\forall i\in \N_N.
\]
Denote by $a:= \max\{a_i\st i\in \N_N\}$ the largest of the Lipschitz constants of the mappings $u_i:\X_i\to \X$ and let $\theta := \frac{1-a}{2L}$. The mapping $d_\theta : (\X\times\Y)\times(\X\times\Y) \to \R$ defined by
\[
d_\theta := d_\X + \theta\,d_\Y
\]
is then a metric for $\X\times\Y$ which is compatible with the product topology on $\X\times\Y$.

\begin{theorem}
The family $\cW_\loc := \{\X\times\Y; (\X_i\times\Y, w_i)\st i\in \N_N\}$ is a contractive local IFS in the metric $d_\theta$ and the graph $G(f^*)$ of the local fractal function $f^*$ associated with the operator $\Phi$ given by \eqref{Phi} is an attractor of $\cW_\loc$. Moreover, 
\be\label{GW}
G(\Phi f^*) = \cW_\loc (G(f^*)),
\ee
where $\cW_\loc$ denotes the set-valued operator \eqref{hutchop} associated with the local IFS $\cW_\loc$.
\end{theorem}

\begin{proof}
We first show that $\{\X\times\Y; (\X_i\times\Y, w_i)\st i\in \N_N\}$ is a contractive local IFS. For this purpose, let $(x_1,y_1), (x_2,y_2)\in \X_i\times\Y$, $i\in \N_N$, and note that
\begin{align*}
d_\theta (w_i(x_1,y_1), w_i(x_2,y_2)) & = d_\X (u_i (x_1), u_i(x_2)) + \theta d_\Y(v_i (x_1,y_1), v_i (x_2,y_2)) \\
& \leq a\, d_\X(x_1, x_2) + \theta d_\Y(v_i (x_1,y_1), v_i (x_2,y_1))\\ 
& \qquad + \theta d_\Y(v_i (x_2,y_1), v_i (x_2,y_2))\\
& \leq (a + \theta L) d_\X(x_1, x_2) + \theta\,s \,d_\Y(y_1,y_2) \\
& \leq q\,d_\theta ((x_1,y_1), (x_2,y_2)).
\end{align*}
Here we used \eqref{scon} and set $q:= \max\{a + \theta L, s\} < 1$. 

The graph $G(f^*)$ of $f^*$ is an attractor for the contractive local IFS $\cW_\loc$, for
\begin{align*}
\cW_\loc (G(f^*)) & = \bigcup_{i=1}^N w_i (G(f^*)\cap \X_i) = \bigcup_{i=1}^N w_i (\{(x, f^*(x)\st x\in \X_i\}\\
& = \bigcup_{i=1}^N \{(u_i (x), v_i(x, f^*(x)))\st x\in \X_i\} = \bigcup_{i=1}^N \{(u_i(x), f^*(u_i(x)))\st x\in \X_i\}\\ 
& = \bigcup_{i=1}^N \{(x, f^*(x)) \st x\in u_i(\X_i)\} = G(f^*).
\end{align*}
That \eqref{GW} holds follows from the above computation and the fixed point equation for $f^*$ written in the form
\[
f^*\circ u_i (x) = v_i (x, f^* (x)), \quad x\in \X_i, \quad i\in \N_N.
\]
%\qedhere
\end{proof}

\section{Computation and Examples}

\subsection{Computational remarks}

The main step in the computation of a fractal function relates in one way or the other to the 
evaluation of the RB operator. We will discuss a discretisation of the RB operator here. 
Note that this discretisation does not involve any numerical approximations but is an 
exact restriction of the full RB operator and will thus (in exact arithmetic) deliver
values of the full RB operator applied to a function.

For computational and visualation purposes we introduce a grid $\X^g \subset \X$ which is a finite 
subset. The numerical computations will then be done for functions $f^g:\X^g \rightarrow \Y$. We 
introduce a restriction $\Phi^g$ of the RB operator $\Phi$ by
$$ \Phi^g f^g\,(x) = \Phi f\,(x), \quad x \in \X^g, \; f^g=f|_{\X^g}.$$
Due to the occurrence of $f_i(u_i^{-1}(x))$, this defines a mapping 
$\Phi^g : \Y^{\X^g} \rightarrow \Y^{\X^g}$ if the grid has the property that $u_i^{-1}(x)\in \X^g$
whenever $x\in u_i(\X_i)\cap \X^g$ for some $\X_i$. If a grid $\X^g$ satisfies this property, we call it
\textbf{admissible}. We then call $\Phi^g$ the \textbf{discrete RB operator} corresponding to the RB
operator $\Phi$ and the grid $\X^g$.

We will now rewrite the discrete RB operator slightly for the case where $\Y=\R$. Note that 
in this case $f^g$ is an element of the finite dimensional vector space $\R^{\X^g} := \R^{|\X^g|}$.
First, we define the (potentially nonlinear) maps
$$ w_i : \R^{\X^g} \rightarrow \R^{\X_i^g}$$
by
$$ w_i(f^g) (x) = v_i(x, f^g(x)), \quad x\in \X_i^g, $$
where $\X_i^g = \X_i\cap \X^g$. Then, we define a linear operator 
$U_i:\R^{\X_i^g}\rightarrow \R^{u_i(\X_i)\cap\X^g}$ by
$$ [U_i f](x) := f(u_i^{-1}(x)), \quad x\in u_i(\X_i)\cap \X^g. $$
$U_i$ is then a sampling operator and we have in particular
$$ [U_i w_i(f^g)](x) = w_i(f^g)(u_i^{-1}(x)) = v_i(u_i^{-1}(x), f^g(u_i^{-1}(x))), 
   \quad x\in u_i(x)\cap \X^g. $$
As the sets $u_i(\X_i)\cap\X^g$ form a partition of $\X^g$ one then has for the discrete RB
operator
$$ \Phi^g f^g = \bigoplus_{i=1}^N U_i\, w_i(f^g). $$

For the special case where $v_i(x, y) = \lambda_i(x) + S_i(x) y$, one introduces the restriction
operator $E_i: \R^{\X^g} \rightarrow \R^{\X_i^g}$ defined by $E_i f(x) := f(x)$ for $x\in \X_i^g$.
The RB operator then is an affine mapping of the form
$$ \Phi^g f^g = \bigoplus_{i=1}^N U_i \lambda_i + U_i S_i E_i f^g, $$
where $S_i$ is the multiplication operator (diagonal matrix) with elements $S_i(x)$. Thus, one has
$$ \Phi^g f^g = \lambda^g + M f^g $$
where the matrix $M$ is factorised in the following way:
$$  M = USE = \begin{bmatrix}U_1 \\ U_2 \\ \vdots \\ U_N\end{bmatrix}
\begin{bmatrix} S_1 & &  & \\ & S_2 &  & \\  & & \ddots & \\ & & & S_N \end{bmatrix}
\begin{bmatrix} E_1 \\ E_2 \\ \vdots \\ E_N \end{bmatrix}. $$
Both matrices $U_i$ and $E_i$ are sampling matrices, i.e., they contain at most one nonzero element
(with value one) in each column. As the matrices $S_i$ are diagonal, one can further simplify the factorisation
as
$$  M = 
\begin{bmatrix} U_1 S_1 U_1^T & & & \\ & U_2 S_2 U_2^T&  & \\ & & \ddots 
  & \\ & & & U_N S_N U_N^T \end{bmatrix}
\begin{bmatrix} U_1 E_1 \\ U_2 E_2 \\ \vdots \\ U_N E_N \end{bmatrix}. $$
Here the matrices $U_i S_i U_i^T$ are square so that the first factor is a diagonal matrix and
the factors $U_i E_i$ are sampling matrices.

One sees that the discrete RB operator can be applied in parallel. However, a difficulty is still 
that in general the evaluation of the sampling operators $U_i E_i$ may require substantial communication
between the processors. This needs to be analysed for each particular case. In some (practically
important) cases, however, one can reduce the amount of communication. This happens when the $X_i$
are uniquely partitioned by some $u_j(\X_j^g)$ in the sense that there exists a partition
$$ \bigcup_{m=1}^M K_m = \N_N$$
such that
$$ \X_i^g = \bigcup_{j\in K_m} u_j(\X_j) \cap \X^g, \quad i\in K_m. $$
From the factorisation above one can derive that in this case the operator $M$ has a block diagonal
structure with $M$ blocks. Furthermore, each block has a factorisation similar to the one above. 
This leads to highly efficient parallel algorithms which will be discussed elsewhere.
We will refer to this case as having a \textbf{local refinement}. Typically, to each block belongs a standard (global)
IFS so that the local IFS consists of $M$ standard ones. The connection between the various IFSs 
is obtained through the choice of the $\lambda_i$ and $S_i$.

\subsection{Example 1: The one-dimensional case with constant $\lambda_i$ and $S_i$}

For this example let $\X=[0,1)$ and $\Y=\R$. Furthermore, let the number $N$ of functions in the local
IFS be even and let $\X_{2j-1}=\X_{2j}=[(j-1)h,jh)$ for $j=1,\ldots,N/2$ where $h=2/N$. Furthermore, let
$$ u_{2j-1}(x) = \frac{x+(j-1)h}{2}\quad \text{and} \quad u_{2j}(x) = \frac{x+jh}{2}, \quad x\in\X_{2j-1}=\X_{2j}. $$
This choice for the mappings $u_i$ implies that $u_i(\X_i) = \left[(i-1)\frac{h}{2},  i \frac{h}{2}\right)$. 
In this first example we choose $ v_i(x,y) = \lambda_i + S_i\, y$, where $\lambda_i, S_i \in \R$ and $|S_i| < 1$, $i = 1, \ldots, N.$ The discrete grid is chosen to be uniform with $h_g=1/N_g$ and where $N_g$ is a multiple of $N$.

One sees that we have here a block structure as discussed at the end of the previous section with $M=2$. Using vector
notation, one gets with $e=(1,\ldots,1)\in \R^{N_g/N}$ the vector
$$ \lambda = (\lambda_1 e,\ldots,\lambda_N e)^T$$
and the matrix
\begin{eqnarray*} 
M & = &S
\begin{bmatrix} S_1 I & & & & & &  \\ & S_2 I &  & & & &  \\ & & S_3 I  & & & &  \\ 
& & & S_4 I   & & &  \\ & & & & \ddots  & & \\ & &  & & & S_{N-1} I & \\ & & &  & & & S_N I 
\end{bmatrix}
\begin{bmatrix} F & & & & \\ F & & & & \\ & F & & & \\ & F & & & \\ & & \ddots & & \\ & & & & F \\ & & & & F 
\end{bmatrix} \\
& = & 
\begin{bmatrix} S_1 F & & & & \\S_2  F & & & & \\ &S_3 F & & & \\ &S_4 F & & & \\ & & \ddots & & \\ 
    & & & & S_{N-1} F \\ & & & & S_N F 
\end{bmatrix},
\end{eqnarray*}
where $F$ is the sampling matrix selecting every second element in the $\X_i$'s. The fractal function is defined on
each domain and there it obeys the fixed point equation
$$ f^g_j = \begin{bmatrix} \lambda_{2j-1} e^T \\ \lambda_{2j} e^T \end{bmatrix} +
           \begin{bmatrix} S_{2j-1} F \\  S_{2j} F \end{bmatrix} \, f^g_j. $$
From these equations one can see that solving this iteratively using the fixed point iteration gives an error of the
order of $O((S_{2j-1}^2 + S_{2j}^2)^{k/2})$ for $k$ iterations.

We selected the $S_i$'s and the $\lambda_i$'s randomly and iterated with the RB operator. The result is
displayed in Figure~\ref{fig:randfracfun}.
\begin{figure}
  \centerline{\includegraphics[width=0.5\textwidth]{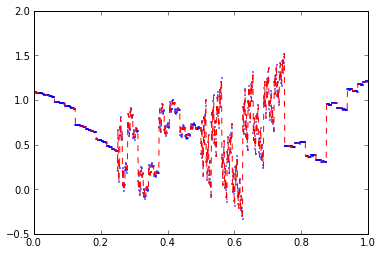}}
\caption{Random 1D fractal function\label{fig:randfracfun}}
\end{figure}
In this case, we chose $N=8$ and thus have four different domains. One can clearly see the different
behaviour on the four domains.

\subsection{Example 2: Interpolating 1D fractal functions}

As before, we choose constant $\lambda_i$ and constant $S_i$. Furthermore, 
assume that the function values at the boundaries of the domains are to be interpolated. From the fixed point
equation one then obtains
$$ \lambda_{2j-1} = (1-S_{2j-1})\, f((j-1)h) $$
and
$$ \lambda_{2j} = (1-S_{2j})\, f(jh). $$
If in addition one would like to have continuity at the midpoint then one needs to require that
$$ (1-S_{2j}-S_{2j-1})\, (f(jh)-f((j-1)h)) = 0. $$
The constants $S_i$ with odd index, $S_{2j-1}$, were chosen randomly and those with even index as
$$S_{2j} = 1-S_{2j-1}.$$ This particular choice implies that the convergence rate is at not any faster than $\sqrt{1/2}$.

If one selects $S_i = 0.5$ for all $i$, a piecewise linear interpolant is obtained. In Figure~\ref{fig:interfracfun}
we have displayed a couple of interpolants for $(x(1-x))^{0.2}$. This shows that some of the interpolants have similar
errors as the piecewise linear interpolant. However, it also shows that at the boundaries some of the interpolants
perform substantially better than the piecewise linear interpolant.
\begin{figure}
  \centerline{\includegraphics[width=0.5\textwidth]{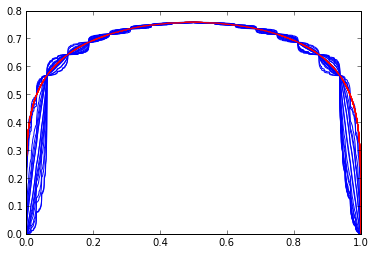}}
\caption{Random 1D interpolating fractal functions \label{fig:interfracfun} for $(x(1-x))^{0.2}$. }
\end{figure}

The evaluation of the RB operator for the interpolation problem converges with the same rate as if one begins 
the iteration at a random point. If, however, one starts the iteration at zero one obtains finite termination
for a finite grid. The number of iterations is of order $O(\log_2(N_g))$ where $N_g$ is the number of
numerical grid points.

\subsection{Example 3: Variable $\lambda_i$ and constant $S_i$}

The main issue here is how to choose the functions $\lambda_i$. From the fixed point
equation $\Phi f = f$ one gets
$$ \lambda_i(x) = f(u_i(x)) - S_i f(x), \quad x \in \X_i.$$
This shows that for any function $f$ and $S_i$ there exists a $\lambda_i$. (See also Theorem \ref{thm3.3}.) 
But this $\lambda_i$ is as complex as the original function and thus there is no gain in
representing $f$ by $\lambda_i$. In some cases, howevever, $f$ cannot be simplified.
In this case one might choose $S_i=0$ and thus 
$$\lambda_i(x) = f(u_i(x)).$$ 

A simple choice for the $\lambda_i$ is: $\lambda_i(x) = \alpha_i + \beta_i x$, for some constants 
$\alpha_i$ and $\beta_i$. As in the case of constants this simple model can also 
lead to rather complicated functions. (See Figure~\ref{fig:ranlinlamconstS}.)
Again one observes a different behaviour on the four different ranges of the $u_i$.
Note, however, that the fractal function $f$ is a linear function of the
$\lambda_i$ so that the dimension of the affine space generated by some $\lambda_i$
has at most as many dimensions as the linear space defined by the vector $\lambda$. 
(In this context, see Theorem \ref{thm3.3} and 
the results in \cite{M97} and \cite{mas2} where this dimension is 
explicitly computed.)
\begin{figure}
  \centerline{\includegraphics[width=0.5\textwidth]{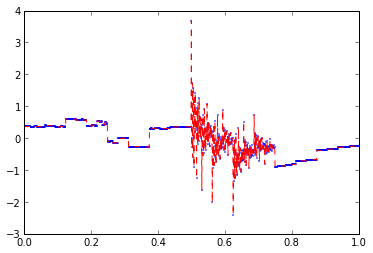}}
\caption{Random fractal function \label{fig:ranlinlamconstS} with affine $\lambda_i$}
\end{figure}

One can also determine $\lambda_i$ such that the resulting fractal function is
interpolatory. If one chooses all $S_i=0.25$, one can select the $\lambda_i$ such
that the resulting fractal function is differentiable at the boundary points between the
domains of the $u_i$ and is therefore a Hermite interpolant at these points. (Cf. also \cite{M97,mas2}.) 
In general, this
function does have discontinuities, in particular, at the midpoints of the domains.
Experiments suggest that the approximation order of this (discontinuous)
interpolant is of third order in the size of the domains. This is the same order
as one would expect from a piecewise quadratic function. An example of the error
curve for the function $\exp(4x)$ can be seen in Figure~\ref{fig:Hermite}. One
can clearly observe that the error is differentiable at the grid points but has some
large discontinuities within the domains $\X_i$.
\begin{figure}
  \centerline{\includegraphics[width=0.5\textwidth]{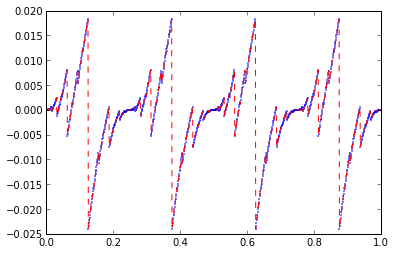}}
\caption{Error of Hermite interpolant of $exp(4x)$ using fractal function 
        \label{fig:Hermite}}
\end{figure}

While linear spaces of $\lambda$ define linear function spaces of fractal 
functions, not every linear function space for $f$ consists of fractal functions.
For this to be the case the function space itself must be self-referential. We define
a linear function space spanned by finitely many functions $\psi_1(x),\ldots,\psi_d(x)$, $d\in \N$, to be 
\textbf{self-referential} if there exist matrices $A_i$ and vectors $b_i$ such that
$$ \psi\circ u_i = b_i + A_i \psi, $$
where $\psi: \X \rightarrow \R^d$ is defined by
$$\psi(x) = (\psi_1(x),\ldots, \psi_d(x))^T.$$
In this case there exists $\lambda_i(x) = \sum_{j=1}^d c_{i,j} \psi_j(x)$ such
that the fractal function defined by the $\lambda_i$ is an element of the function
space.
Prominent examples of such function spaces include polynomials and scaling
functions. More generally, the condition of self-referantiabilty for bases is found
in the subdivision schemes of computer graphics.

\section{Polynomial Fractals}

\subsection{The Taylor series}

In the following we will investigate the fractal nature of the graph of
polynomials 
\[
  p : [0,1] \rightarrow \R.
\] 
This research is done with a view to the development of efficient
numerical algorithms. In the future we will consider complex-valued polynomials and also real-analytic analytic
functions.

For our purposes, we denote by $\ell_0$ the space of all real-valued sequences having only finitely many terms not equal to zero. As is common practice, we endow $\ell_0$ with the ``norm'' $\|a\|_0 := |a_0|^0 + \cdots |a_n|^0$, $a\in \ell_0$. Here, we defined $0^0 := 0$. Furthermore, we denote by $\ell_p$, $p >0$, the space of all real-valued sequences $\R^\N\ni x:=\{x_n\st n\in \N_0\}$ such that
\[
\sum_{n=0}^\infty |x_n|^p < \infty.
\]
Note that for $p \geq 1$, 
\[
\|x\|_{\ell^p} := \left(\sum_{n=0}^\infty |x_n|^p\right)^{1/p}
\]
defines a norm making $\ell_p$ into a Banach space. For $0 < p <1$, the function
\[
d_p (x,y) := \sum_{n=0}^\infty |x_n - y_n|^p
\]
defines a metric making $\ell_p$ into a complete topological vector space which is not normable. In this setting, $\ell_0$ may be thought of as $\ell_p$ where $p\to 0+$.

Let $a\in \ell_0$ be a finite sequence and define a function
$v:[0,1] \rightarrow \ell_1$ with components 
\[
   v_k(x) = \frac{x^k}{k!}.
\] Then the function $p$ given by 
\[
  p(x) = \sum_{k=0}^\infty a_k v_k(x)
\] 
is a polynomial and $a_k$ is the value of its $k$-th derivative at
zero. One can see that any derivative of $p$ satisfies 
\[
  \frac{d^k p(x)}{dx^k} = \sum_{j=0}^\infty a_{k+j} v_j(x).
\] 
This motivates the introduction of a function
$f: [0,1] \rightarrow \ell_0$ with components
\[
f_k(x) = \frac{d^k p(x)}{dx^k}.
\] 
As $p$ is a polynomial only a finite
number of components $f_k$ are not equal to zero. One can now reformulate
the Taylor series of $p$ at any point $x$ as
\[
  p(x+t) = f(x)^T v(t).
\] 
Similar formulas for all the derivatives of $p$ may be obtained in a similar fashion.
This can all be stated using the matrices $A(x)$ and $V(t)$ defined by
\[
  [A(x)]_{ij} = [f(x)]_{i+j}
\] 
and 
\[
  [V(t)]_{ij} = \begin{cases} [v(t)]_{i-j}, \quad&  \text{if}\; i\geq j \\
                0 & \text{else}, \end{cases}
\] 
respectively. In the following choose indices $i,j=0,1,\ldots$ to always start at
zero. Note that $A(x)$ is a Hankel matrix and $V(t)$ a Toeplitz
matrix. From the above one can show that the Taylor series for all the
derivatives takes the form 
\[
   f(x+t) = A(x) v(t) = V(t) f(x).
\]

The infinite matrix $A(x)$ is a Hankel matrix and the anti-diagonals
take the values $f_i(x)$. As $f(x)$ only has a finite number of
elements (corresponding to $p$ being a polynomial) there exists some nonnegative
integer $M$ so that 
\[
   f_M(x) \neq 0 \quad \text{and} \quad f_k(x) = 0 \; \text{for $k>M$}.
\] 
Now let $A_n(x) \in \R^{n\times n}$ denote the principle submatrix of
$A(x)$. It follows that $A_{M+1}(x)$ is invertible left upper
triangular with antidiagonal elements $a_M$. Consequently, the
generalised inverse $A(x)^+$ of $A(x)$ is a matrix which has zero
elements except for a principle $(M+1)\times (M+1)$ block $(A(x)^+)_M$
which is 
\[
   (A(x)^+)_M = A_M(x)^{-1}.
\]

\begin{lemma}
Let $A$ be an infinite upper triangular Hankel matrix with
$A_{M,0} \neq 0$ and $A_{k,0} = 0$ for all $k>M$. Let $[B]_M$ denote the
principal $(M+1) \times (M+1)$ block of any matrix $B$ and let $A^+$
denote the Moore-Penrose inverse of $A$. Then $[A]_M$ is nonsingular and
\[ 
[A^+]_M = [A]_M^{-1} 
\] 
is a lower triangular Hankel matrix. Moreover,
\[ 
A^+ = \sum_{k=0}^{\infty} (I-\frac{1}{a_M}P_M A)^k   \frac{1}{a_M}P_M A.  
\]
\end{lemma}

\begin{proof}
Let $P_M$ be the permutation with $(P_M x)_j = x_{M-j}$ for
$j=0,\ldots M$ and $(P_M x)_j = x_j$ for $j> M$. Then $[P_M A]_M$ is a
lower triangular Toeplitz matrix with nonzero diagonal. Consequently it
and also $[A]_M$ are invertible. If $B$ is the matrix which is zero
except for the principle submatrix $[B]_M$ and such that
$[B]_M = [A]_M^{-1}$ then one can show that $AB$ is an infinite matrix
with $[AB]_M = [I]_M$ and zero elsewhere. (Here, $I$ denotes the identity matrix.)
From this one easily confirms the four defining criteria of a
Moore-Penrose Inverse.

As $P_M A$ is regular lower triangular Toeplitz with diagonal elements
$a_M$ the matrix $L = I-\frac{1}{a_M} P_M A$ is also lower triangular
Toeplitz with zero diagonal. It follows that $L^M = 0$ and thus one can
obtain the inverse of $[A]_M$ using the geometric series for $I+L$. This 
leads to the stated formula.
\end{proof}

One can get an explicit formula for the inverse. For simplicity we omit the $x$. As
$M$ is the degree of the polynomial and $P_M$ is the reversal
permutation of the first $M+1$ elements, the geometric series converges
as any term with $k\geq M$ is zero.

The matrix $V(t)$ also has a nice structure and one can see that 
\[
   V(t) = \exp(t \sigma)
\] 
where $\sigma$ is the forward shift matrix given by
$[\sigma]_{i,j}=\delta_{i,j-1}$. Consequently one obtains 
\[
  V(t)^{-1} = V(-t) = \exp(-t\sigma)
\] 
and thus
\[
  f(x) = V(-t) f(x+t).
\] 
Hence, the operators $V(t)$ form a group with 
\[
  V(t) V(s) = V(t+s), \quad V(0) = I, \quad V(t)^{-1} = V(-t).
\] 
In addition, we also obtain that $V(s)^T v(t) = v(t+s)$.
\subsection{Self-referentiality of $v$ and $f$}
So far we have considered the properties of $v$ under translations. We
will now consider dilations. The dilations in the $x$-space are defined
by mappings $l_x$ of the form.
\[
  l_x(t) = (1-s)x + st.
\] 
By definition one has 
\[
  v(sx) = D_s v(x)
\] 
where $D_s = \operatorname{diag}(s^k)_{k=0,\ldots,\infty}$. Then the
map on $\R\times \ell_1 \rightarrow \R \times \ell_1$ given by 
\[
 w(x,y) := (sx, D_s y)
\] 
satisfies $w(x,v(x)) = (sx, D_s v(x)) = (sx, v(sx))$ and thus $w$
leaves the graph $\left\{(x,v(x))\right\}$ invariant. More generally,
one has

\begin{align*}
  v(l_x(t)) & = v(x+s(t-x))  \\
            & = V(x)^T v(s(t-x)) \\
            & = V(x)^T D_s v(t-x) \\
            & = V(x)^T D_s V(-x)^T v(t)
\end{align*}

Then the mapping $w$ defined by 
\[
  w(t,y) := (l_x(t), V(x)^T D_s V(-x)^T y)
\] 
satisfies 
\[
  w(t, v(t)) = (l_x(t), v(l_x(t)))
\] 
and consequently $w$ leaves the graph of $v$ invariant. A similar
observation has also been reported in a forthcoming publication Barnsley et al.
\cite{BHM}.

Next, we like to find functions $w$ under which the graph of $f$ is invariant. To this end, 
consider 
\[
  w(t,y) = (l_x(t), A(x) D_s A(x)^+ y).
\] 
Recall from above that $A(x) v(t) = f(x+t)$. Therefore, one
concludes that 
\[
  [A(x)^+ f(t)]_k = \begin{cases} v_k(t-x), \quad \text{for $k\leq M$}\\
                                  0, \quad \text{for $k > M$}. \end{cases}
\] 
An argument similar to the one given in the previous example yields
\[
  w(t,f(t)) = (l_x(t), f(l_x(t)),
\] 
implying that $w$ leaves the graph of $f$ invariant. According to \cite{BHM} we call
the mappings $w$ which leave a polynomial invariant \textbf{fractels}. For more details and fundamental 
properties of fractels, we refer the reader to the upcoming publication \cite{BHM}.
\subsection{Affine IFSs for given polynomials}
Here we combine two fractels $w$ from the previous section to form an
IFS. The infinite matrix 
\[
   W_s(x) = A(x) D_s A(x)^+
\] 
has the following properties:

\begin{itemize}
\itemsep1pt\parskip0pt\parsep0pt
\item
  $W_s(x)$ is of rank $M+1$;
\item
  Most eigenvalues of $W_s(x)$ are thus equal to zero. The nonzero eigenvalues are
  $1,s,s^2,\ldots, s^M$;
\item
  $W_s(x)$ is lower triangular (and the eigenvalues are on the diagonal).
\end{itemize}

\noindent
It is possible to use the fractels introduced in the last section but due to the occurrence of
the eigenvalue 1, the fixed point of the resulting IFS is not unique and
typically depends on the starting point. Note that if linear mappings
are used with all eigenvalues less than zero the only fixed point is the zero
function. Thus in this case one needs eigenvalues of value 1. Such an approach
may be suited for the case of projective spaces, here however we consider affine spaces. Therefore, 
we replace the linear $w$ from last section by 
\[
  w(t, y) := (l_x(t), (W_s(x) - \theta e_0 e_0^T) y + \theta f_0(x) e_0)
\] 
for some $\theta \in [0,1]$. In practice the choice $\theta=0.5$ was
very stable but sometimes lead to slow convergence. Choosing $\theta=0$
was faster but less stable. This behaviour will be investigated further and the results reported elsewhere.

The particular choice of affine function also leaves the graph of $f$ invariant as
\[
  (W_s(x) - \theta e_0 e_0^T) f(x) + \theta f_0(x) e_0 = W_s(x) f(x).
\] 
Note that we used $e_0 := (1,0,\ldots)$.

For illustrative purposes, let us consider $x\in[0,1]$ and define an IFS $\{[0,1]; w_1, w_2\}$
consisting of two functions $w_1$ and $w_2$ which correspond to the Taylor expansion at
$x=0$ and at $x=1$. Furthermore, we let us choose $s=0.5$. Then
\[
  w_0(x,y) = (0.5 x, (W_{0.5}(0)-\theta\, e_0 e_0^T)y + \theta a_0 e_0)
\] 
and 
\[
  w_1(x,y) = (0.5(x+1), (W_{0.5}(1)-\theta\, e_0 e_0^T)y + \theta b_0 e_0)
\] 
where we have $a_0 = f_0(0)$ and $b_0=f_0(1)$.

\section{Algorithms}

In this section we present some algorithmic aspects which are mostly motivated
by the Collage Theorem. We first consider convex optimisation, then grids and
finally subdivision. Here we only provide a rough outline. A more detailed
treatment is under development.

\subsection{Collage fitting}

In this section a new kind of approximant for the solution of elliptic
problems is introduced. We call this approximant \emph{collage fit}. Like the 
common Ritz method this approximation is shown to be quasi-optimal.  
Let in the following $H$ be a Hilbert space and $a(\cdot,\cdot)$ be a symmetric $H$-elliptic
form. We consider here the problem of determining 
$$
   \hat{u} = \argmin_{u\in V} \Psi(u).
$$
where $\Psi(u) = \frac{1}{2} a(u,u) - b(u)$ and $b$ is a continuous linear functional on $H$.
Let $V_N\subset H$ be an $M$ dimensional linear subspace of $H$. 
The widely used \emph{Ritz method} provides an approximation $\hat{u}_N\in V_N$ to $\hat{u}$ defined by
$$  \hat{u}_N := \argmin_{u\in V_N} \Psi(u). $$
It can be shown that the Ritz method minimises the energy norm of the error $\hat{u}_N-\hat{u}$, i.e.,
$$ 
  \|\hat{u}_N - \hat{u}\|_E \leq \|u_N - \hat{u}\|_E, \quad \text{for all $u_N\in V_N$}
$$
where $\|v\|_E = \sqrt{a(v,v)}$. A consequence of the H-ellipticity is that the energy 
norm is equivalent to the $H$-norm, i.e., there exist $c_1, c_2 > 0$ such that
\begin{equation}\label{eq:BDequivalence}
  c_1 \|v\| \leq \|v\|_E \leq c_2 \|v\|, \quad \text{for all $v\in H$.}
\end{equation}
It follows directly that the Ritz approximation is quasi-optimal, and in particular
$$ \|\hat{u}_N - \hat{u}\| \leq \frac{c_2}{c_1}\, \|u_N - \hat{u}\|, \quad \text{for all $u_N\in V_N$}. $$
%%%%

We define $V_N$ as a set of fractal functions as follows. 
Let $F(\cdot;\alpha) : H\rightarrow H$ denote a family of RB operators (as 
defined in a previous section) parameterised by a parameter vector $\alpha\in\R^M$.
We will assume that the RB operators are contractive, i.e., that 
$$ \|F(u;\alpha) - F(v;\alpha)\| \leq c \|u-v\|, \quad \text{for all $u,v\in H$} $$
for some constant $c\in (0,1)$. We will also assume a stronger condition, namely
that 
$$ \gamma := \frac{c\, c_2}{c_1} < 1.$$
Finally, we will assume that 
$F(u;\alpha)$ is a linear function of $(u,\alpha) \in V\times \R^M$. 
These assumptions hold for commonly used RB operators. The fixpoint $u_\alpha$ of an RB operator 
$F(\cdot;\alpha)$ is a fractal function. As approximation set for our elliptic problem we consider
$$V_N = \{u_\alpha \mid \alpha\in\R^M, u_\alpha = F(u_\alpha;\alpha)\}.$$
As $F$ is linear in $(u,\alpha)$ the set $V_N$ is a finite-dimensional linear space, and, in addition
that $F$ can be decomposed as
$$ F(u;\alpha) = F(u;0) + F(0;\alpha).$$

It follows that $F(u;\alpha)-F(v;\alpha)=F(u;0)-F(v;0)$ and thus all the $F(\cdot;\alpha)$
are contractive with a constant $c$ independent of $\alpha$.

For the following let $W:= \{F(0;\alpha) \mid \alpha \in \R^M\}$. Note that $W$ is a linear space and
define the affine space
$$ W(u) := F(u;0) + W.$$
We now introduce the operator $G:H\rightarrow H$ by 
$$ G(u) := \argmin_{v\in W(u)} \Psi(v) $$
where $\Psi(v)$ is the quadratic form defined previously.

\begin{samepage}
\begin{proposition}
\label{prop:61}
\begin{itemize}
\item Let $\Psi$ be an $H$-elliptic quadratic form which defines an energy norm $\|\cdot\|_E$ for which there exist
$c_1,c_2>0$ such that $c_1\|v\|\leq\|v\|_E\leq c_2\|v\|$ for all $v\in H$.
\item Let $F(u;\alpha)=F(u;0)+F(0;\alpha)$ define an RB operator which is contractive with constant $c$
such that $c < c_1/c_2$. 
\item Let $G(u) = \argmin_{w\in W(u)} \Psi(w)$.
\end{itemize}
Then the so defined operator $G$ is contractive and 
$$ \|G(u)-G(v)\| \leq \gamma \|u-v\|$$
where $\gamma = c c_2/c_1$.
\end{proposition}
\end{samepage}

\begin{proof}
As $G(u)$ is the best approximation in $W(u)$ to $\hat{u}$ one can show that $\hat{u}-G(u)$ is
orthogonal to the space $W$ with respect to the scalar product $a(\cdot,\cdot)$ and the same
holds for $\hat{u}-G(v)$. Thus $G(u)-G(v)$ is orthogonal to $W$ in the same scalar product. It
follows that $\|G(u)-G(v)\|_E$ is the distance between $W(u)$ and $W(v)$ in the energy norm. As
this distance is the minimum distance between any point of $W(u)$ and any point of $W(v)$ one has
in particular
\begin{align*}
  c_1 \|G(u)-G(v)\| &\leq \|G(u)-G(v)\|_E \\
                    & \leq \|F(u;0)-F(v;0)\|_E \\
                    & \leq c_2 \|F(u;0)-F(v;0)\| \\
                    & \leq c_2 c \|u-v\|
\end{align*}
and thus $\|G(u)-G(v)\| \leq \gamma \|u-v\|$. 
\end{proof}

One then has: 
\begin{corollary}[Existence of collage fit $\tilde{u}_N$]
\label{cor:62}
Let $G$ be as in Proposition~\ref{prop:61}. Then there exists a unique $\tilde{u}_N\in V_N$ such
that $\tilde{u}_N=G(\tilde{u}_N)$.
\end{corollary}
\begin{proof}
 As $G$ is contractive there exists a unique $\tilde{u}_N\in H$ such that $\tilde{u}_N=G(\tilde{u}_N)$.
 As $\tilde{u}_N\in W(\tilde{u}_N)$ there exists an $\alpha\in\R^M$ such 
that $\tilde{u}_N=F(\tilde{u}_N,\alpha)$. Thus $\tilde{u}_N\in V_N$.
\end{proof}

Thus the \emph{collage fit} $\tilde{u}_N\in V_N$ is defined to be the fixpoint of $G$. Note that 
this is an approximation of $\hat{u}$ which is in $V_N$, it is,however, in general different from 
the Ritz approximation $\hat{u}_N$. Nonetheless it is also a quasi-optimal approximation:
\begin{proposition}[quasi-optimality of collage fit]
  Let $\tilde{u}_N$ be the collage fit for the quadratic form $\Psi$ as 
  defined in Corollary~\ref{cor:62}. If all the assumptions of this corollary hold and if
  $\gamma$ and $c$ are as defined in this corollary, then one has
  $$\|\tilde{u}_N - \hat{u}\| \leq \frac{1/c+1}{1/\gamma-1}\, \|u_N - \hat{u}\|, \quad 
    \text{for all $u_N\in V_N$.} $$
\end{proposition}
\begin{proof}
Let $u_N\in V_N$ and $\alpha\in\R^M$ such that $u_N=F(u_N;\alpha)$. As $\tilde{u}_N$ minimises the
energy norm in $W(\tilde{u}_N)$ one has
\begin{align*}
  c_1 \|\tilde{u}_N-\hat{u}\| & \leq \|\tilde{u}_N-\hat{u}\|_E \\ 
  & \leq \|F(\tilde{u}_N;\alpha)-\hat{u}\|_E \\
  & \leq c_2 \|F(\tilde{u}_N;\alpha)-\hat{u}\|_E.
\end{align*}
By the triangle inequality and contractivity of $F(\cdot;\alpha)$ one has
\begin{align*}
  \|F(\tilde{u};\alpha)-\hat{u}\| & \leq \|F(\tilde{u};\alpha)-F(\hat{u};\alpha)\| 
                + \|F(\hat{u};\alpha)-F(u_N;\alpha)\| + \|u_N-\hat{u}\| \\
  & \leq c\|\tilde{u}-\hat{u}\|+c\|\hat{u}-u_N\|+\|u_N-\hat{u}\|
\end{align*}
and thus
$$\|\tilde{u}-\hat{u}\|\leq \frac{c_2}{c_1}((1+c)\|u_N-\hat{u}\|+c\|\tilde{u}-\hat{u}\|). $$
The claimed inequality follows directly. 
\end{proof}
%%%%

We then compute the collage fit $\tilde{u}_N$ iteratively using the fixpoint algorithm for 
$G$:
\vskip 5pt
\noindent
\textbf{Collage Fitting Algorithm}

\begin{itemize}
\itemsep1pt\parskip0pt\parsep0pt
\item First choose some $u^{(0)}$.
\item Then repeat for all $k=0,1,2,\ldots$
  $$u^{(k+1)} = G(u^{(k)})$$
\end{itemize}

The algorithm converges because of the contractivity of the operator $G$. Applications of this
algorithm include quasi-optimal approximations of $L_2$ functions by classes of fractal functions.
In practice we found that these approximations are very close to the best $L_2$ approximations.
Other applications are the computation of fractal approximations to the solution of Fredholm
integral equations of the first kind using Tikhonov regularisation. Finally, this approach can
also be used to solve elliptic PDEs numerically with fractal functions. More details on these
applications will be provided in a forthcoming paper.

\subsection{Evaluation of functions on grids}
Grids are very important objects for numerical computations. They are a
collection of points where during the computations one needs to compute
the unknown function in order to get the value at the points one is
interested in.

In the simple case of local interpolation one requires just neighboring
points. However, if one would like to solve a PDE one needs a whole
field.

In the end, however, the values of interest are a function of a certain
collection of values at other points. This is a type of
self-referentiality and we now proceed to define self-referential grids. This approach
is based on a upcoming paper by Barnsley et al. \cite{BHM} on the computation of
function values. Here we only consider the simple case discussed above
where we have an IFS with two functions and our functions are defined
over $[0,1]$.

To this end, consider $x\in [0,1]$. Then $x$ has a numerical representation of the form 
$$
x=0.d_1 d_2 d_3 \ldots d_J,\quad d_i \in \{0,1\}.
$$
Let functions $l_i : [0,1] \rightarrow [0,1]$ be given by
\[ 
l_0(x) := x/2, \quad \text{and} \quad l_1(x) := (x+1)/2. 
\]
Then $x$ is defined by the recursion
\[  
x^{(0)} = 0, \quad x^{(k+1)} = l_{d_{J-k}}(x^{(k)}), \quad k=0,\ldots, J. 
\]
If $f(x)$ is the vector of derivatives of some polynomial evaluated at
$x$ then one may use the recursion
\[ 
y^{(0)} = f(0), \quad y^{(k+1)} = W_{d_{J-k}} y^{(k)} + b_{d_{J-k}} 
\]
to obtain $y = f(x)$. This is essentially the method of function
evaluation discussed in \cite{BHM}.

Here we note that in order to obtain the value of $f$ at the point $x$
one requires the values of $f$ on all $x^{(k)}$. This is the ``grid''
required to determine $f(x)$. This ``grid'' is nothing else but the
path of the shift function $\sigma$ starting at point $x$ where 
\[
  \sigma(0.d_1 d_2 d_3 \ldots) = 0.d_2 d_3 \ldots.
\]

More generally, we define a \textbf{self-referential grid} $\gamma$ as a
finite set of points $\{0,1\} \subset \gamma \subset [0,1]$ such that 
\[
  \gamma \subset l_0(\gamma) \cup l_1(\gamma).
\] We now have

\begin{proposition}

A self-referential grid $\gamma$ is invariant under $\sigma$, i.e.,
\[ 
\sigma(\gamma) \subset \gamma.
\]
\end{proposition}
\begin{proof}
As $\gamma$ is self-referential there exists for every $x\in \gamma$ a
$z\in \gamma$ such that $x=l_i(z)$, for some $i\in\{0,1\}$.

If $z= 0.d_1 d_2 \ldots$ then $l_0(z) = 0.0d_1 d_2 \ldots$ and
$l_1(z) = 0.1 d_1 d_2 \ldots$. In both cases one has $z=\sigma(x)$. Thus,
 $\sigma(x)\in \gamma$ and we have shown that for any $x\in\gamma$, 
$\sigma(x)\in\gamma$. 
\end{proof}

Hence, we now can define for any finite set $M\subset [0,1]$ a
self-referential grid $\gamma_M = \bigcup_k \sigma^k(\gamma)$. If we
know the IFS we can then determine the values on all the points of
$\gamma_M$ recursively (as outlined above). In particular, one then also obtains the
values on the set $M$. One could use this for multiscale modelling where
one models fine scale behaviour on just a small subset of a very fine
grid and uses self-referentiality to get the overall solution.

\subsection{Subdivision schemes}

Subdivision schemes are widely used in computer graphics for modelling curves and
surfaces. An introduction and survey of the mathematics of subdivision schemes
can be found in~\cite{CavDM91,DynL02,MicP89,PraM87,PraBP02}. A subdivision scheme is
a collection of mappings (called refinement rules)
$R_k: V_k \rightarrow V_{k+1}$ between linear spaces $V_k$ of real
functions defined on nested meshes (at most countable sets of isolated points)
$N_0 \subset N_1 \subset \cdots \subset \R^s$. 

Iterated function systems (and LIFSs) provide a rich source of subdivision schemes.
For example, consider an IFS with $\X=[0,1)$, $N=2$ and
$$  u_1(x)  = \frac{1}{2} x \quad \text{and} \quad u_2(x) = \frac{1}{2} + \frac{x}{2} $$
for $x\in [0,1)$. Furthermore, let $v_i(x,y)$ be continuous for $i=1,2$ with
$v_1(1,y) = v_2(0,y)$ for $y\in\R$. One then obtains a subdivision scheme with meshes
$N_k = 2^{-k}\Z_{2^k}$ and $\Z_{2^k} = \{0,\ldots, 2^k-1\}$ by choosing the refinement
rules $R_k : \R^{N_k} \rightarrow \R^{N_{k+1}}$ to be
$$
  (R_k f) (\xi) =
\begin{cases}
  v_1(2\xi, f(2\xi)), \quad \xi \in [0,1/2) \cap N_{k+1} \\
  v_2(2\xi-1, f(2\xi-1)), \quad \xi \in [1/2,1) \cap N_{k+1}.
\end{cases}
$$
Note that these rules are well defined as $2\xi \in N_k$ in the first case and $2\xi-1\in N_k$
in the second case. An important question regards the convergence of subdivision schemes to
continuous functions. In the examples generated by LIFSs one obtains this convergence 
directly from the convergence of the LIFS itself.

Subdivision algorithms -- like LIFSs -- are used to generate the values of 
graphs of functions. Generalising the concept of polyomial fractals
discussed in the previous section, one now may obtain LIFSs from the common
subdivision schemes, see also the book by Prautzsch et al~\cite{PraBP02}
for a different angle of this discussion based on B\'ezier splines. More specifically, 
Micchelli and Prautzsch \cite{PraM87,MicP89} discuss refinement algorithms which use a
refined basis based on uniform subdivision. They present the unified
structure of a large class of smoothing methods. In particular, they
show that the obtained curves are \textbf{uniformly refinable} or
self-referential in the sense that the curve may be patched together
from scaled subcomponents of itself. This fundamentally defines a
local IFS and, in particular, generalizes methods used for Bezier curves
which are based on polynomials.

%\color{red}Also include: Weimar/Warren: Solve PDEs with subdivision\color{black}

\section{Conclusions and Final Remarks}

We have demonstrated that fractal functions defined by local iterated function systems can be used for computations.
In fact, many known methods including piecewise polynomial approximation and wavelets and more generally subdivision
schemes can be described within the fractal framework because the underlying components (the polynomials and wavelets) 
have a fractal nature. 

While this fractal nature has been observed in particular in the subdivision and wavelet literature, one observes 
that even some of the newest numerical approximation 
schemes do have a fractal nature. As an illustration thereof, we consider here the QTT (quantized tensor train) method. It considers functions
which can be represented by matrix products of the form
$$ 
   f(x) = \sum_{\alpha_1,\ldots,\alpha_d} g_1(i_1,\alpha_1) \prod_{k=2}^{d-1} g_k(\alpha_k,i_k,\alpha_{k+1})\, 
            g_d(\alpha_d, i_d),
$$
where $x$ has the binary representation
$$ x = \sum_{k=1}^d i_k 2^{k-1}. $$
The approximation of functions using their binary digits in this way was motivated by the work on high-dimensional
approximation and quantum mechanics. QTT was introduced by Oseledets in~\cite{Ose09}. The summation ranges of the
indices $\alpha_k=1,\ldots,r_k$ are defined by the tensor train ranks $r_k$. For computational efficiency it is
important that these ranks are small. Except for special cases (the exponential function, trigonometric functions
and piecewise polynomials) little is known~\cite{Kho11,Gra10,Ose13} about which functions can be approximated by QTT 
functions with low ranks. We briefly remark that fractals admit such a representation. This demonstrates
that the fractal framework considered here is also useful for the analysis of the QTT method.

Consider in particular a fractal function defined by
\begin{align*}
  f(x/2) & = \lambda_1 + S_1 f(x) \\
  f(x/2+1/2) & = \lambda_2 + S_2 f(x).
\end{align*}
Let $x$ have the binary representation with binary digits $i_1, i_2, \ldots$ as above and let
$$ y = \sum_{k=1}^{d-1} i_{k+1} 2^{k-1}. $$
Then the recursion for the fractal function can be rewritten as
$$ f(x) = \lambda_{i_1+1} + S_{i_1+1}\, f(y) $$
or in matrix form as
$$ f(x) = 
\begin{bmatrix} 1 & 0 \end{bmatrix}
\begin{bmatrix} S_{i_1+1} & \lambda_{i_1+1}\\ 0 & 1 \end{bmatrix} \begin{bmatrix} f(y) \\ 1 \end{bmatrix}. $$
If one now iterates this for $f(y)$ one gets the factorisation
$$ f(x) = \begin{bmatrix} 1 & 0 \end{bmatrix} \prod_{k=1}^d
       \begin{bmatrix} S_{i_k+1} & \lambda_{i_k+1}\\ 0 & 1 \end{bmatrix} 
       \begin{bmatrix} f(0) \\ 1 \end{bmatrix}. $$
This provides an explicit QTT representation for the fractal function $f(x)$ and shows that these fractal functions
have QTT rank 2. Note, however, that we have only considered a function class with 4 parameters $\lambda_i$ and $S_i$.
Rank 2 QTT functions allow the parameters to depend on the levels or position of digits. This can also be discussed
in the fractal framework and will be considered in future works as will local IFSs.

\section*{Acknowledgments}
The third author wishes to thank the Mathematical Sciences Institute of The Australian National University for its kind hospitality and support during his visit in May 2013.

\end{document}